\documentclass[12pt,reqno]{amsart}
 \usepackage{amssymb}
 \usepackage{graphicx}

\newtheorem{theorem}{Theorem}

\theoremstyle{remark}
\newtheorem{remark}{Remark}
\newtheorem{definition}[theorem]{Definition}

\def\BBZ {{\mathbb Z}}

\def\BBR {{\mathbb R}}
\def\BBC {{\mathbb C}}

\begin{document}

\title[Removability of H\"older graphs]
{Removability of H\"older graphs for continuous Sobolev functions}

\author[N. Tecu]{Nicolae Tecu}
\email[N. Tecu]{nicolae.tecu@yale.edu}
\address[N. Tecu]%
{Department of Mathematics \\
 Yale University \\
 New Haven, CT 06520 \\
 USA}

\date{June 2010}

\begin{abstract}
We characterize the removability of H\"older-$\alpha$ graphs with respect to continuous Sobolev $W^{1,2}$ functions. For $\alpha > 2/3$ these graphs are removable, while for $\alpha <2/3$ there exist graphs which are not removable.
\end{abstract}

\maketitle

\section{Introduction}

In this paper we characterize the removability of H\"older graphs
with respect to continuous Sobolev functions. The problem of finding
necessary and sufficient metric conditions for removability has been
studied in previously, most of the time in conjunction with the
closely related removability of sets for quasiconformal mappings. In
this paper we address the first problem, while making reference to
the second one. However, as one will see, the characterization we
provide is not yet complete for the class of quasiconformal
mappings.

The setting of all the theorems is $\BBR^2$. $\Omega$ is an open
set. In the following $W^{1,p}(\Omega)$ will denote the functions in
$L^p(\Omega)$ whose distributional partial derivatives are also
functions in $L^p(\Omega)$. This means that $u\in W^{1,p}(\Omega)$
if $u\in L^{p}(\Omega)$ and there are functions $\partial_j u\in
L^{p}(\Omega),j=1,2$ such that
\begin{equation}
  \label{eq:sobnt}
  \int_{\Omega}u\partial_j\psi dx = -\int_{\Omega}\psi\partial_j u
  dx
\end{equation}
for all test functions $\psi\in C_0^\infty(\Omega)$ and all $1\leq
j\leq n$.

\begin{definition}
  A compact set $K$ is called $W^{1,p}$-removable for continuous functions in $\BBR^2$
 , if any function, continuous in $\BBR^2$ and in $W^{1,p}(\BBR^2\setminus
  K)$ belongs to $W^{1,p}(\BBR^2)$.
\end{definition}

A function $f:\Omega \rightarrow \tilde{\Omega}$ is a quasiconformal
mapping if $f$ is a homeomorphism, it is absolutely continuous on
almost every line parallel to the coordinate axes (ACL),
differentiable almost everywhere and $\mbox{max}_{\alpha}
\partial_\alpha f(x) \leq C \mbox{min}_{\alpha} \partial_\alpha
f(x)$ holds almost everywhere. The min and max are over directions
$\alpha$.

\begin{definition} A compact set $K\subset U$ is quasiconformally
removable in $\BBR^2$, if any homeomorphism of $\BBR^2$ which is
quasiconformal on $\BBR^2\setminus K$ is quasiconformal on $\BBR^2$.
\end{definition}

A quick look over the definitions reveals that removability is about
the absolute continuity on almost every line and that if a set is
$W^{1,2}$-removable for continuous functions then it is also
removable for quasiconformal mappings. It is not known whether the
converse holds also. Obviously, a set which is not removable for
quasiconformal mappings will also be non-removable for continuous
$W^{1,2}$ functions.

The quasiconformal removability problem has been extensively
studied. F. Gehring proved that any set $K$ of $\sigma$-finite
length is quasiconformally removable (\cite{G}). It is not difficult
to see that any set $K$ of positive area is non-removable. In terms
of Hausdorff measure this is the most one can say.

Some of the most significant results have been proven by P. W. Jones and
S. Smirnov (\cite{JS00}) which give sufficient geometric conditions for
continuous Sobolev removability. These have been improved by P. Koskela
and T. Nieminen in \cite{KN05}.
On the non-removability side, R. Kaufman proved that there are non-removable
graphs (\cite{K84}) and in a further paper on the topic he showed that for each
$\alpha < 1/2$ there is a non-quasiconformally removable H\"older $\alpha$ graph (\cite{K86}).

By a {\it H\"older $\alpha$ graph} we mean the graph of a function $f:[0,1]\rightarrow \BBR$
which satisfies the condition $|f(x)-f(y)|\leq C |x-y|^{\alpha}$ for all $x,y\in [0,1]$.
More generally, by a {\it H\"older h graph} we mean the graph of a function
$f:[0,1]\rightarrow \BBR$ which satisfies the condition $|f(x)-f(y)|\leq h(|x-y|)$
for all $x,y\in [0,1]$, where $h:\BBR_+\rightarrow \BBR_+$ is a homeomorphism with $h(0)=0$.
Other aspects of the two problems have been studied by P. Koskela (\cite{Ko99}), R. Kaufman and
J-M. Wu (\cite{KW96}), J-M. Wu (\cite{W98}), 
C. Bishop (\cite{B94}) (the list is not exhaustive).

In this paper we prove the following theorems.

\begin{theorem}\label{theorem1} Let $p>1, p'$ it's conjugate. The graph of a Hoelder h function is removable for continuous Sobolev $W^{1,p}$ functions if the following condition holds:
\begin{equation}
\int_0^1 \left(\frac{t}{h^{-1}(t)}\right)^{p'} dt < \infty
\end{equation}
In particular the condition $\alpha > \frac{p}{2p-1}$ is sufficient for continuous Sobolev $W^{1,p}$ removability. For $p=2$ we get $\alpha > 2/3$ as sufficient condition (hence also for quasiconformal mappings).
\end{theorem}

\begin{theorem}\label{theorem2}
 Let $p>1$. Then for each $\alpha < \frac{p}{2p-1}$ there is a H\"older $\alpha$ graph $\Gamma$ which is not removable for continuous $W^{1,p}$ functions. In particular if $p=2$ we get $\alpha <2/3$.
\end{theorem}
We do not know what happens for $\alpha = \frac{p}{2p-1}$.
\begin{theorem} \label{theorem3}
For each $\alpha<1/2$ there is a graph which is non removable for quasiconformal mappings.
\end{theorem}

We have not been able to construct a quasiconformally non removable graph for $\alpha\in[1/2, 2/3)$.

Theorem \ref{theorem3} has already been proven by R. Kaufman in
\cite{K86}. For a comparison of the two proofs and their limitations
see the paragraphs after the proof of theorem \ref{theorem3}.

For the sake of completeness we give a theorem of Peter W. Jones
(personal communication) which shows that for any Sobolev (or
quasiconformally) non-removable graph $\Gamma$ for $\alpha>1/2$ any
interesting function $F$ must satisfy $|\nabla F|\rightarrow \infty$
as $(x,y)\rightarrow \Gamma$.

\begin{theorem}\label{theorem4} If $F$ is a function which is ACL off a H\"older $\alpha$ graph $\Gamma$, $\alpha>1/2$ ,and satisfies $|\nabla F|\leq 1$, then $F$ is globally ACL.
\end{theorem}

In the following section we prove the sufficiency result. The last
section is dedicated to the non-removability results and to theorem
\ref{theorem4}.

{\bf Acknowledgments}I would like to thank my advisor, Peter W. Jones, for proposing this problem and for the discussions about it.

\section{Proof of sufficiency}

In this section we prove Theorem \ref{theorem1}. We do this by applying a result of Jones and Smirnov (\cite{JS00}).

Consider a simply connected domain $\Omega$ with a marked point $z_0$ and it's Whitney decomposition $\{Q\}$. We denote by $Sh(Q)$ the shadow of the square $Q$ through the family of hyperbolic geodesics starting at $z_0$ and accumulating on the boundary of $\Omega$. In other words, $Sh(Q)$ is the set of accumulation points on $\partial \Omega$ of hyperbolic geodesics that start at $z_0$ and pass through $Q$. We denote by $s(Q)$ the euclidean diameter of $Sh(Q)$ and by $l(Q)$ the side length of the square $Q$. In this section and the next we will use $\approx$ and $\lesssim$ for equality and inequality respectively up to \textit{universal} constants. We will keep track of constants important for the computation separately.

Following  \cite{JS00} we have the following theorem:
\begin{theorem}\label{JonesSmirnov}
If for some $p > 1$ a simply connected domain $\Omega\subset \BBR^2$ satisfies
\begin{equation}
\sum_{Q} \left(\frac{s(Q)}{l(Q)}\right)^{p'}l(Q)^2 <\infty
\end{equation}
then $\partial \Omega$ is removable for continuous Sobolev $W^{1,p}$ functions ($1/p+1/p' =1$).
\end{theorem}
\begin{remark} The sum is over the Whitney squares which are at some finite (euclidian) distance to the boundary. The precise distance is not important.
\end{remark}

\begin{proof}[Proof of theorem \ref{theorem1}]
We consider a H\"older h graph $\Gamma_f$ and the Whitney decomposition $\{Q\}$ of the complement of the graph. For each $n\in \BBZ, n \geq 0$ let $S_n$  denote the collection of Whitney squares $Q$ which are at height $\approx 2^{-n}$ above or below the graph $\Gamma_f$. Let $LR$ be the collection of squares which are to the left or right of the graph (the projection of any of these squares on the x-axis does not intersect $[0,1]$).

We then have trivially
\begin{equation*}
\sum_{Q\in LR} \left(\frac{s(Q)}{l(Q)}\right)^{p'}l(Q)^2 <\infty
\end{equation*}
and need only estimate the rest of the sum:
\begin{eqnarray*}
\sum_{Q\notin LR} \left(\frac{s(Q)}{l(Q)}\right)^{p'}l(Q)^2  = \sum_n\sum_{S_n} \left(\frac{s(Q)}{l(Q)}\right)^{p'}l(Q)^2
\end{eqnarray*}
For each of the squares in $S_n$ the following properties hold:
\begin{itemize}
\item[(a)] $l(Q) \approx h^{-1}(2^{-n})$
\item[(b)] $s(Q) \approx 2^{-n}$
\end{itemize}
Between height $2^{-n}$ and $2^{-n+1}$ there are at most $\approx 2^{-n}/h^{-1}(2^{-n})$ such squares vertically and at most $\approx 1/h^{-1}(2^{-n})$ horizontally.

All these imply
\begin{eqnarray*}
\sum_n\sum_{S_n} \left(\frac{s(Q)}{l(Q)}\right)^{p'}l(Q)^2 &\lesssim& \sum_n \left(\frac{2^{-n}}{h^{-1}(2^{-n})}\right)^{p'} \left(h^{-1}(2^{-n})\right)^2\frac{2^{-n}}{\left(h^{-1}(2^{-n})\right)^2} \\
&\lesssim& \sum_n \left(\frac{2^{-n}}{h^{-1}(2^{-n})}\right)^{p'} 2^{-n}\\
&\lesssim& \int_0^1 \left(\frac{t}{h^{-1}(t)}\right)^{p'} dt < \infty
\end{eqnarray*}
By theorem \ref{JonesSmirnov} we obtain the desired conclusion.
\end{proof}

\section{Examples of non-removable graphs}

In this section we construct non-removable graphs. We will follow the philosophy of \cite{K84}. Compare also with \cite{K86}.

\begin{proof}[Proof of theorem \ref{theorem2}]
We construct a map $u:\BBR^2 \rightarrow \BBR$ which has the following properties.
\begin{itemize}
\item[(a)] it is not identically 0 and it nonnegative
\item[(b)] it is continuous
\item[(c)] it has partial derivatives almost everywhere and $\frac{\partial u}{\partial x} = 0$ a.e.(off a graph $\Gamma$).
\item[(d)] it is absolutely continuous on almost every line off $\Gamma$, but not globally ACL.
\item[(e)] $\nabla u \in L^p$ and  $\alpha< \frac{p}{2p-1}$.
\end{itemize}

The function $u$ above will be the limit of a sequence of functions $u_n$, each of which has support in $[0,1]^2$, is continuous and is differentiable almost everywhere. It will be given by $u_n(x,y) = \int_0^x A_n(t,y)dt$ for some function $A_n$. 

Together with $u_n$ we also construct a sequence of sets $\{\Gamma_n\}$ with $\cap_n\Gamma_n = \Gamma$. The function $A_n$ above will be supported on $\Gamma_n$.

The first step of the procedure is to consider $A_1:[0,1]^2\rightarrow \BBR$ which is constant in the x-variable and $A_1(x,0) = A_1(x,1) = 0, \forall x\in [0,1]$. We normalize it to have a maximum 1. Pick two large even numbers $M$ and  $N$ (how we pick these will become clear later).

Assume we are now at level $n$. The function $A_n$ is supported on a finite collection of rectangles of sides $l_n, \tilde{l}_n$ respectively. $\Gamma_n$ is the union of these rectangles. Inside each rectangle, $A_n$ is constant as a function of the $x$-variable. Consider one of these rectangles, $R$.

We divide the $y$ side of $R$ in $N$ pieces and the $x$ side in
$2(N-1)M$ pieces. What we get is a collection of small rectangles.
In this mesh we mark as black {\it some} rectangles of sides
$l_{n+1} = \frac{l_n}{2(N-1)M}$ and $\tilde{l}_{n+1} =
\frac{2\tilde{l}_n}{N}$. Each of these black rectangles will be made
of two small rectangles (basically two small ones, one on top of the
other). On each column there will be only one black rectangle. Every
second column will have no black rectangle in it. Finally, the black
rectangles will be placed in the pattern $\Lambda \Lambda \ldots
\Lambda$ (see Figure).

If we denote by $I_{i}$ the projections on the $y$-axis of the black rectangles we should have $\cup_i I_{i} = $ the whole $y$-side of $R$. In addition $I_i \cap I_{i+1} \neq \emptyset $.

The finite  collection $\{I_i\}_{1,N-1}$ is a cover of the $y$-side of $R$ so we can find a partition of unity $\phi_i$ subordinated to this cover such that $\sum_i \phi_i(y) = 1$ (this is easy, everything is finite). In fact at most two functions will be non zero for any $y$.

Now define $A_{n+1}$ to be zero whenever $A_n$ is zero. Set $A_{n+1}(x,y) = 0 $ for $(x,y)$ outside of the black rectangles. In each black rectangle the function $A_{n+1}$ is constant as a function of $x$. We define it such that:
\begin{eqnarray}
A_{n+1}(x,y) l_{n+1} M = \phi_{n+1,i}(y) A_n(x,y)l_n \label{DefAn}
\end{eqnarray}
if $(x,y)$ is in a black rectangle with projection $I_i$ on the $y$-axis.
This will ensure that $A_{n+1}$ is supported on the black rectangles.

One can think of $u_{n+1}$ as describing a measure on each line. In this view $A_{n+1}$ is essentially the 'mass' density of $u_{n+1}$. Condition (\ref{DefAn}) ensures that the mass on each line is preserved when going from level $n$ to $n+1$. More rigorously, for each $y$:
\begin{eqnarray}
\int_{x_{R,in}}^{x_{R,out}}A_n(t,y)dt &=& \int_{x_{R,in}}^{x_{R,out}} A_{n+1}(t,y)dt  \\
&=& \sum_{B-\mbox{black rect}} \int_{x_{B,in}}^{x_{B,out}}A_{n+1}(t,y)dt
\end{eqnarray}
Here $x_{R, in}, x_{R, out}$ are the first and the last $x$ coordinates for which $(x,y)$ is in rectangle $R$.

$\Gamma_{n+1}$ is defined as the union of all black rectangles from all rectangles $R\subset \Gamma_n$.

\includegraphics{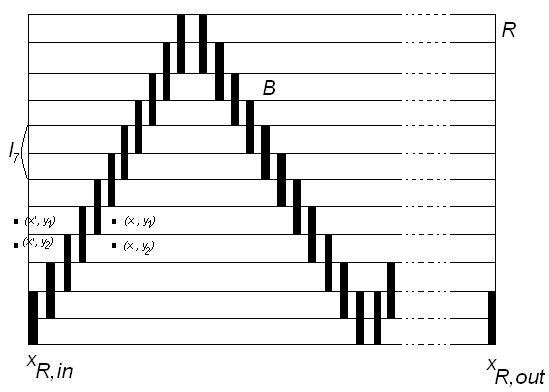}

Since  $u_n(x,y) = \int_0^x A_n(t,y)dt$ we have $\frac{\partial u_n}{\partial x} = 0$ outside $\Gamma_n$. In addition, $u_{n+i}(x,y) = u_n(x,y)$ for all $(x,y)\notin \Gamma_n, \forall i\geq1$.

We also have the relation:

\begin{eqnarray}
|u_{n+1}(x,y)-u_n(x,y)| = |\int_0^x A_{n+1}(t,y)-A_n(t,y)dt|
\end{eqnarray}

Because the function $u_{n+1}$ is no different than $u_n$ outside $\Gamma_n$, the supremum of these quantities is reached when $x\in[x_{B,in},x_{B,out}]$, where $B\subset \Gamma_n$ is the black rectangle for which $(x,y)\in B$. We then have:
\begin{eqnarray}
|u_{n+1}(x,y)-u_n(x,y)| \lesssim \frac{A_n(x,y)l_n}{M} 
\end{eqnarray}
By the defining relation of $A_n$ , (\ref{DefAn}), we get:
\begin{eqnarray}
\sup_{x,y} A_n(x,y) &\lesssim& \frac{1}{M^{n}l_n}
\end{eqnarray}
and
\begin{equation*}
\frac{\partial A_{n+1}(x,y)}{\partial y} l_{n+1} M  = \frac{\partial \phi_{n+1,i}(y)}{\partial y} A_n(x,y)l_n + \phi_{n+1,i}(y) \frac{\partial A_{n}(x,y)}{\partial y} l_n \nonumber
\end{equation*}
We keep in mind that $\sup |\phi_{n+1,i}|\leq 1$ and its support is
of length $\tilde{l}_{n+1}$ and get
$\sup|\partial\phi_{n+1,i}/\partial y|\lesssim
\frac{1}{\tilde{l}_{n+1}}$. Then we compute recursively and get:
\begin{equation}
\sup_{x,y}|\frac{\partial A_{n}(x,y)}{\partial y}| \lesssim \frac{1}{ M^{n} l_{n}\tilde{l}_n}\label{derAn}
\end{equation}
We may thus conclude:
\begin{eqnarray}
|u_{n+1}(x,y)-u_n(x,y)| \lesssim \frac{1}{M^{n+1}} 
\end{eqnarray}
Since $M>1$ $u_n\rightarrow u$ uniformly. $u$ will thus be
continuous. In addition, $u(x,y) = u_n(x,y)$ for $(x,y)\notin
\Gamma_n$ so u satisfies all the properties we set out to fulfill,
except the last one.

We will now estimate $b_{n+1} = \sup |\frac{\partial
u_{n+1}(x,y)}{\partial y}|$  for $(x,y)\in
\Gamma_n\setminus\Gamma_{n+1}$. Consider $(x,y_1),(x,y_2)$ which
satisfy these conditions and pick $x'$ such that $(x',y_1),
(x',y_2)$ lie just outside and to the left of the black rectangle of
level $n$ in which we picked $(x,y_1),(x,y_2)$. Then
\begin{eqnarray}
|u_{n+1}(x,y_1) - u_{n+1}(x,y_2)|&\leq& |u_{n+1}(x',y_1) - u_{n+1}(x',y_2)| +\nonumber \\
&+&|u_{n+1}(x,y_1) - u_{n+1}(x',y_1) + u_{n+1}(x',y_2)- u_{n+1}(x,y_2)| \nonumber\\
&\leq& |u_{n}(x',y_1) - u_{n}(x',y_2)| + \nonumber \\
&+& |\int_{x'}^x A_{n+1}(t,y_1)dt - \int_{x'}^x A_{n+1}(t,y_2)dt| \nonumber \\
&\leq&  b_n |y_2-y_1| +  |\int_{x'}^x A_{n+1}(t,y_1) - A_{n+1}(t,y_2)dt| \nonumber \\
&\leq&b_n |y_2-y_1| +  \int_{x'}^x \sup|\frac{\partial A_{n+1}(t,y)}{\partial y}| |y_2-y_1|dt \nonumber \\
&\leq&b_n |y_2-y_1| + \frac{C|y_2-y_1|M l_{n+1}}{M^{n+1} l_{n+1}\tilde{l}_{n+1}} \nonumber \\
&\leq&\left( b_n + \frac{C}{M^{n}\tilde{l}_{n+1}}\right)|y_2-y_1|
\end{eqnarray}
We may now write
\begin{equation*}
b_{n+1}\leq b_n +  \frac{C}{M^{n}\tilde{l}_{n+1}}
\end{equation*}
$C$ above denotes a universal constant!
Now, we have $l_n < \frac{1}{N^nM^n}, \tilde{l}_n = \frac{2^n}{N^n}$ so we have
\begin{equation*}
b_{n+1}\lesssim \sum_{i=2}^{n+1}  \frac{N^i}{2^{i}M^{i-1}} \lesssim N\sum_2^{n+1} \frac{N^{i-1}}{M^{i-1}}\lesssim \frac{N}{N/M - 1} \frac{N^{n+1}}{M^{n+1}}
\end{equation*}
if $N> M$. In case $N\leq M$ we have $b_{n+1}\lesssim N$.

We define $\Gamma = \cap_n \Gamma_n$. It is easy to see this is the graph of a function.

Keeping in mind that $|\frac{\partial u}{\partial y}| \leq b_n $ for $(x,y)\in \Gamma_n\setminus \Gamma_{n-1}$ we can write:
\begin{eqnarray}
\int_{\Gamma^C}|\nabla u|^p dxdy &\lesssim& \sum_n M^nN^nl_n\tilde{l}_n b_n^p \lesssim C(N,M)\sum_n \frac{2^{n}}{N^n} \frac{N^{np}}{M^{np}} \nonumber \\
&\lesssim& C(N,M) \sum_n \left(\frac{2N^{p-1}}{M^p}\right)^n
\end{eqnarray}
So the $L^p$ integral of $\nabla u$ is going to be finite if and only if
\begin{equation}\label{condition1}
\frac{2N^{p-1}}{M^p} <1
\end{equation}
The constant $C(N,M)$ is finite for all $N, M$. If $N \leq M$ the
integral is trivially finite. In this case we actually obtain a
Lipschitz function $u$. We will come back to this later.

We have up to this point constructed a function $u$ which satisfies all the properties we wanted. This function is ACL on $\Gamma^C$, where $\Gamma$ is the graph of a function. To ensure this is a H\"older $\alpha$ graph we should have:
\begin{equation}
\tilde{l}_n\leq \left(\frac{l_n}{M}\right)^\alpha \forall n >0 \label{condition2}
\end{equation}
We write $M=2^a, N=2^b$. $M$ is actually a constant (large, but nevertheless a constant) so we can disregard it above. Conditions (\ref{condition1}) and (\ref{condition2}) become:
\begin{eqnarray*}
b(p-1)+1 &<& ap \\
1+\alpha + a\alpha  &\leq& b(1-\alpha)
\end{eqnarray*}
The second relation becomes:
\begin{eqnarray*}
\alpha(1+a+b) &\leq& b-1 \Rightarrow \alpha \leq \frac{b-1}{1+a+b}< \frac{b-1}{1+b+\frac{b(p-1)+1}{p}} \\
&<& \frac{p(b-1)}{p+b(2p-1)+1} \rightarrow \frac{p}{2p-1} \mbox{ as }b\rightarrow \infty.
\end{eqnarray*}
We thus see that if we pick $a,b$ large enough $\Gamma$ will be a H\"older $\alpha$ graph. And this is possible for all $\alpha < \frac{p}{2p-1}$.
\end{proof}

\begin{proof}[Proof of theorem \ref{theorem3}]
We have $p=2$. For $\alpha <1/2$ one can take $N < M$ above which
implies $|\nabla u|< C$. Consider the function $F:\BBC\rightarrow
\BBC, F(z) = z + u(x,y)$. This function maps horizontal lines to
horizontal lines and on each one of them the function is strictly
increasing. Hence it is a homeomorphism. Since  $|\nabla u|< C$ the
Beltrami differential has norm $<k<1$ and so $F$ is quasiconformal
off $\Gamma$. On the other hand $F$ is not globally ACL (because $u$
isn't), hence not globally quasiconformal.
\end{proof}

We have not been able to construct a quasiconformally non removable
graph for $\alpha\in[1/2, 2/3)$. The trouble is that one is forced
to pick $N\geq M$ and that leads to a function $u$ with Beltrami
coefficient $\mu_F$ which satisfies $||\mu_F||_{\infty} = 1$.  We
hoped we could solve the Beltrami equation for $\mu_F$ and correct
$F$, but our attempts were not successful, neither using the
generalization to the Measurable Riemann Mapping theorem for
exponential distortion, nor using Lehto's theorem (\cite{AIM09},
chapter 20). In our model it is possible to make $F$ a
homeomorphism, but very difficult to make it quasiconformal off the
graph.

In his proof, R. Kaufman defined a function $F = Az + \int
\frac{1}{z-w}d\mu(w)$ for some constant $A$ and measure $\mu$
supported on a graph. The gradient of the Cauchy transform turned
out to be bounded off $\Gamma$ which meant that for large $A$ $F$
was a homeomorphism. For $\alpha \geq 1/2$ any interesting function
can't be Lipschitz anymore (see below) so the Cauchy transform can't
be Lipschitz which in turn means that most likely $F$ will not be a
homeomorphism.

\begin{proof}[Proof of theorem \ref{theorem4}]
It was proven by A. S. Besicovitch and H. D. Ursell (\cite{BU}) that
\begin{equation*}
Hdim(\Gamma)\leq 2-\alpha
\end{equation*}
An easy consequence is that for almost every horizontal line $l$:
\begin{equation} \label{HdimGammaLine}
Hdim(\Gamma\cap l)\leq 1-\alpha
\end{equation}
(See e.g. \cite{M95}, theorem 7.7) Consider one such line $l$ and on
which $F|_{\Gamma^C}$ is absolutely continuous. Let $\epsilon > 0$.
We want $\delta$ such that if $\sum_i|x_i-y_i|<\delta$ then $\sum_i
|F(x_i)-F(y_i)|<\epsilon$. We have two types of intervals
$(x_i,y_i)$ (on line $l$): some which don't contain any point of the
graph $\Gamma$ (we call them 'good') and some which do (we call them
'bad').

Without loss of generality we may consider only the case when the
bad intervals are an open cover of $\Gamma \cap l$. Since $\alpha >
1/2$ there exists $\tau > 0$ such that $2\alpha > 1+\tau$.
 By (\ref{HdimGammaLine}) we may pick an open cover $\{(\tilde{x}_j,\tilde{y}_j)\}$ of $\Gamma\cap l$
 such that $|\tilde{x}_j-\tilde{y}_j|\leq \inf_i|x_i-y_i|$  (this is actually a minimum) and such that
\begin{equation*}
\sum_{j}|\tilde{x}_j-\tilde{y}_j|^{1-\alpha+\tau}  < 1
\end{equation*}
We also must have for each $j: (\tilde{x}_j,\tilde{y}_j)\subset
(x_i,y_i)$ for some $i$ and
\begin{equation*}
\sum_j |\tilde{x}_j-\tilde{y}_j|\leq \sum_{bad}|x_i-y_i| <\delta
\end{equation*}
Finally, we also have that the intervals composing
$(x_i,y_i)\setminus \cup_j (\tilde{x}_j,\tilde{y}_j)$ are actually
in the complement of $\l\cap \Gamma$ so we can lump them with the
good intervals. We can redefine the "bad" collection of intervals to
be $\{(\tilde{x}_j, \tilde{y}_j)\}$. The collection of all the
intervals "good" and "bad" is a refinement of the one we started
with.

Since $\Gamma$ is a H\"older $\alpha$ graph and  $|\nabla F| \leq 1$
we get $|F(\tilde{x}_j)-F(\tilde{y}_j)| \leq
|\tilde{x}_j-\tilde{y}_j|^{\alpha}$. By the definition of $\tau$ we
have $p = \frac{\alpha-\tau}{1-\alpha} > 1$. Take $q$ to be its
conjugate and set $\alpha_1 = \frac{1-\alpha+\tau}{p}, \alpha_2 =
\frac{1}{q}$. Then $\alpha  = \alpha_1+\alpha_2$. We may now write:
\begin{eqnarray}
\sum_{j} |F(\tilde{x}_j)-F(\tilde{y}_j)|&\leq& \sum_{j}|\tilde{x}_j-\tilde{y}_j|^{\alpha}\leq\left(\sum_{j}|\tilde{x}_j-\tilde{y}_j|^{p\alpha_1}\right)^{1/p}\left(\sum_{j}|\tilde{x}_j-\tilde{y}_j|^{q\alpha_2}\right)^{1/q}\nonumber \\
&<&
\left(\sum_{j}|\tilde{x}_j-\tilde{y}_j|^{1-\alpha+\tau}\right)^{1/p}\left(\sum_{j}|\tilde{x}_j-\tilde{y}_j|\right)^{1/q}\leq
\delta^{1/q}\label{relation1}
\end{eqnarray}

Since the good intervals are in the complement of $\Gamma$ we use
$|\nabla F| \leq 1$ to get
\begin{equation*}
\sum_{good} |F(x_i)-F(y_i)|\leq \sum_{good}|x_i-y_i| < \delta
\end{equation*}
We are done once we pick $\delta$ small enough.
\end{proof}


\begin{thebibliography}{99}

\bibitem{AIM09} Astala K, Iwaniec T., Martin, G. \textit{Elliptic partial differential equations and quasiconformal mappings in the plane} Princeton University Press, 2009.

\bibitem{BU} Besicovitch A. S., Ursell, H. D. \textit{Sets of fractional dimensions (V): On dimensional numbers of some continuous curves} Journal of the London Mathematical Society s1-12 (1), 18-25
\bibitem{B94} Bishop, C., \textit{Some homeomorphism of the sphere conformal off a curve} Ann. Acad. Sci. Fenn. Ser. A I Math. {\bf 19}, (1994), 323-338.

 \bibitem{G}
  F. W. Gehring
  \textit{The definitions and exceptional sets for quasiconformal mappings},
  Ann. Acad. Sci. Fenn. Ser. A I Math \textbf{281} (1960), 1-28.

 \bibitem{JS00} Jones, Peter W. and Smirnov, Stanislav S.,
\textit{ Removability theorems for Sobolev functions and
quasiconformal maps} Ark. Mat. {\bf 38} (2000), 263-279;

\bibitem{K84} Kaufman, R. \textit{Fourier-Stiljies coefficients and continuation of functions} Ann. Acad. Sci. Fenn, Ser. A I Math {\bf 9} 1984, 27-31

\bibitem{K86} Kaufman, R. \textit{Plane curves and removable sets} Pacific Journal of Math. {\bf 125} (1986), no. 2, 409-413.

\bibitem{KW96} Kaufman, R., Wu, Jang-Mei, \textit{On removable sets for quasiconformal mappings} Ark. Mat., {\bf 34}, (1996), 141-158.

\bibitem{Ko99} Koskela, P., \textit{Removable sets for Sobolev spaces} Ark. Mat., {\bf 37} (1999), 291-304

\bibitem{KN05} Koskela, Pekka; Nieminen, Tomi \textit{Quasiconformal removability and the quasihyperbolic metric} Indiana Univ. Math. J. {\bf 54}  (2005),  no. 1, 143--151

\bibitem{M95} Mattila, P., \textit{Geometry of sets and measures in euclidean spaces} Cambridge University Press 1995.

\bibitem{W98} Wu, J-M., \textit{Removability of sets for quasiconformal mappings and Sovolev spaces} Complex Variables Theory Apl. {\bf 37}, (1998), 491-506.
\end{thebibliography}
\end{document}